\documentclass[12pt, reqno]{amsart}
\usepackage{amsmath, amsthm, amscd, amsfonts, amssymb, graphicx, color}
\usepackage[bookmarksnumbered, colorlinks, plainpages]{hyperref}

\newtheorem{theorem}{Theorem}[section]
\newtheorem{lemma}[theorem]{Lemma}
\newtheorem{proposition}[theorem]{Proposition}
\newtheorem{corollary}[theorem]{Corollary}
\theoremstyle{definition}
\theoremstyle{notation}
\newtheorem{definition}[theorem]{Definition}
\newtheorem{notation}[theorem]{Notation}
\newtheorem{question}[theorem]{Question}
\newtheorem{example}[theorem]{Example}

\newtheorem{problem}[theorem]{Problem}
\theoremstyle{remark}
\newtheorem{remark}[theorem]{Remark}
\numberwithin{equation}{section}

\begin{document}
\setcounter{page}{1}

\title[Some Problems in Functional Analysis]{Some Problems in Functional Analysis Inspired by Hahn Banach Type Theorems}

\author[M.A. Sofi]{M.A. Sofi$^1$ $^{*}$}

\address{$^{1}$ Department of Mathematics, University of Kashmir, Hazratbal Srinagar - 190 006, India.}
\email{\textcolor[rgb]{0.00,0.00,0.84}{aminsofi@rediffmail.com}}



\date{Received: xxxxxx; Revised: yyyyyy; Accepted: zzzzzz.
\newline \indent $^{*}$ Corresponding author}
 \maketitle
\begin{center}
Dedicated to Professor T. Ando with affection.
\end{center}
\begin{abstract}
As a cornerstone of functional analysis, Hahn Banach theorem constitutes an indispensable tool of modern analysis where its impact extends beyond the frontiers of linear functional analysis into several other domains of mathematics, including complex analysis, partial differential equations and ergodic theory besides many more. The paper is an attempt to draw attention to certain applications of the Hahn Banach theorem which are less familiar to the mathematical community, apart from highlighting certain aspects of the Hahn Banach phenomena which have spurred intense research activity over the past few years, especially involving operator analogues and nonlinear variants of this theorem.

For a discussion of a whole lot of issues related to the Hahn Banach theorem not treated in this paper, the best source is a famous survey paper by Narici and Beckenstein \cite{nb} which deals, among other things, with the different settings witnessing the validity of the Hahn Banach theorem.
\end{abstract}

\textbf{Contents}
\begin{enumerate}
\item[\ref{sa}.] What is known to be folklore.\\
\item[\ref{sb}.] Unconventional applications of Hahn Banach theorem.\\
\item[\ref{sc}.] Some less known aspects of Hahn-Banach Theorem.\\
\item[\ref{sd}.] Hahn Banach-extension property as a Finite Dimensional property.\\
\item[\ref{se}.] Hilbert spaces determined via Hahn Banach phenomena.\\
\item[\ref{sf}.] Intersection of balls and extendibility of maps.\\
\end{enumerate}
\section{What is known to be folklore}\label{sa}
\begin{definition}\label{d1}
\textnormal{ Hahn Banach Theorem (HBT) (Classical, Real Case)\\
Given a real normed linear space $X$, a subspace $Y$ of $X$, a continuous linear functional $g$ on $Y$, there exists a continuous linear functional $f$ on $X$ such that $\| f \| =\| g \|$.}
\end{definition}
\begin{definition}\label{d2}
\textnormal{ HBT(Complex case)\\
Same as above with $X$ being a linear space over $\mathbb{C}$ and $g$ being complex linear. Then $F$ can also be chosen to be complex linear.}
\end{definition}	
\begin{corollary}\label{c1}
For $\mathbb{R}^{n}$- valued continuous linear maps $g$ on $Y$, there always exist continuous linear maps $f$ on $X$ extending $g$  $($not necessarily with preservation of norms$)$.
\end{corollary}	
\begin{remark}\label{r1}
For the above extension to be norm-preserving, one has to take $\ell_{n}^{\infty} = (\mathbb{R}^{n},\| \|_{\infty})$ as the range space. Interestingly, converse is also true!\\
(See the discussion in Remark \ref{r5} of Section \ref{sc}).
\end{remark}
\begin{proposition}\label{p1}
Given a norm $\| \|$ on $ \mathbb{R}^{n}$, then $(\mathbb{R}^{n}, \| \|)$ has norm-preserving extension property if and only if $(\mathbb{R}^{n},\| \|)=\ell_{n}^{\infty}$, isometrically.
\end{proposition}
\begin{remark}\label{r2}
(i): HBT also holds for the more general class of locally convex spaces which include normed spaces as a very special case. This is a consequence of the fact that the locally convex topology of such spaces is given by a family of seminorms and that the classical version of HBT also holds in the more general setting of norms being replaced by seminorms.\\
\textit{Remark} 1.6. (ii):  Existence of Hahn-Banach extension is a typically {\it locally convex} phenomenon! In fact, the spaces 
$L_{p}[0,1]$ are known to admit no continuous linear functionals as long as $0<p<1$ (See \cite[Theorem 12.31]{abo}). This is 
a special case of a more general situation: a quasi Banach space $X$ is locally convex if and only if every closed 
subspace of $X$ has the Hahn Banach extension property (HBEP)  i.e., for each  $g\in M^*,\exists \, f \in X^* $  
 such that $f=g$  on $M$ (See \cite[Theorem 4.8]{kn1}). Further, it turns out that the stated equivalence is still valid 
if the (HBEP) is assumed to hold for those closed subspaces M of $X$ for which both M and $\frac{X}{M}$ are Banach 
spaces. On the other hand, there are (incomplete) nonmetrizable non-locally spaces which have the (HBEP) \cite{gs}.
\end{remark}
\begin{notation}\label{n1}
\textnormal{  Throughout this paper, we shall let $X$, $Y$, $Z$ denote Banach spaces, unless
otherwise stated. We shall use the symbol $X^{\ast}$ for the dual of $X$ whereas $B_{X}$ shall be used
for the closed unit ball of $X$:
$$ B_{X}=\{x\in X,\| x \|\leq 1\}.$$
We shall also make use of the following notation:\\
$L(X,Y)$, Banach space of bounded linear operators from $X$ into $Y$.
$$ c_0(X)=\{(x_n)\}\subset X:x_n\rightarrow 0\}$$
$$ \ell_p[X]=\{(x_n)\subset X:\sum_{n=1}^{\infty}\mid\langle x_n,f\rangle\mid^p<\infty,\forall~f\in X^{\ast}\}$$
$$ \ell_p\{X\}=\lbrace(x_n)\subset X:\sum_{n=1}^{\infty}\| x_n\|^p<\infty\rbrace$$
Clearly, $\ell_p\{X\}\subset\ell_p[X],\forall~1\leq p <\infty$. Further, the indicated inclusion is continuous when these spaces are equipped with natural norms defined by:
$$ \epsilon_p((x_n))=\sup_{f\in B_{X^\ast}}{\bigg(\sum_{n=1}^{\infty}\mid\langle x_n,f\rangle\mid^p\bigg)}^{1/p},~(x_n)\in\ell_p[X].$$
$$ \pi_p((x_n))={\bigg(\sum_{n=1}^{\infty}\| x_n\|^p\bigg)}^{1/p},~(x_n)\in\ell_p\{X\}.$$
We shall also have occasion to use p-summing maps in the sequel. 
Thus given $T\in L(X,Y)$, we shall say that $T$ is {\it (absolutely) 
p-summing} $(1\leq p < \infty)$ $\big( T\in \Pi_p(X,Y)$ if $\forall \{x_n\}\in \ell_p[X]$, it follows that 
$\{T(x_n) \} \in \ell_p\{X\}$. By the open maping theorem, this translates into the finitary condition: $\exists 
c > 0$ s.t.\\
$$(*) {\big(\sum_{i=1}^n|| T(x_i)||^p\big)}^{1/p} \leq c.{ \sup_{f\in B_{X^\ast}}\bigg\{ 
\big(\sum_{i=1}^n{|<x_i,f>|}^p\big)}^{1/p}\bigg\}, \,\, \forall{(x_i)}^n_{i=1}\subset X, n\geq 1.$$
The infimum of all such a c's appearing above and denotes by $\pi_p(T)$ is called the p-{\it summing norm} of T, making $\Pi_p(X,Y)$ into the Banach space.}
\end{notation}
\section{Unconventional applications of HB-theorem}\label{sb}
\begin{enumerate}
\item[\ref{s2.1}] Unified approach to fundamental theorems of functional analysis.
\item[\ref{s2.2}] Proof of certain classical theorems of analysis.
(Riemann mapping theorem/Muntz theorem/Existence of Green’s functions).
\item[\ref{s2.4}] Invariant version of HB-theorem.
\item[\ref{s2.3}] Existence of Banach limits.
\item[\ref{s2.5}] Existence of Banach measures on $\mathbb{R}$ and $\mathbb{R}^{2}$.
\item[\ref{s2.6}] Existence of non-measurable sets.
\end{enumerate}

\subsection{Unified approach to fundamental theorems of functional analysis}\label{s2.1}
\hspace*{\fill} \\
In a first course on functional analysis, one learns the four fundamental theorems of
functional analysis: uniform boundedness principle, Hahn Banach theorem, open mapping
theorem (and the closed graph theorem). One also quickly learns that each of these theorems is
approached independently of each other in that whereas HBT is typically proved via Zorn’s
lemma (or a variant there of), the main idea in the proof of the other theorems is provided by a
category argument involving certain families of sets appearing in these proofs. However, it is
possible to provide a unified treatment to all these fundamental theorems via the Hahn Banach
extension theorem, something which is of great pedagogical value.

The main tool in this approach is provided by the following lemma due to W.Roth \cite{rh} which
can be derived from the classical version of the Hahn Banach theorem quoted in the very
beginning of this article.

\begin{lemma}\label{l1}
Let X be a normed space and let $A\subset X$ and $S\subset X^{\ast}$ be such that S is pointwise bounded on $X$ but not uniformly bounded on A. Then there exists $f\in X^{\ast}$ such that $f$ is unbounded on A. In fact, $f$ can be chosen to be the functional:
$$f(x)=\sum_{n=1}^{\infty}\alpha_{n}f_{n},\alpha_{n}\geq 0,\sum_{n=1}^{\infty}\alpha_{n}=1,f_{n}\in S,x\in X.$$
\end{lemma}

The strategy for proving the four fundamental theorems consisting in using Lemma \ref{l1} in the proof of (a special case of) the uniform boundedness principle (UBP) to the effect that in a normed space X,
\begin{enumerate}
\item[(a)] bounded subsets are precisely those which are weakly bounded and
\item[(b)] $weak^{\ast}$ bounded subsets of $X^{\ast}$ are norm- bounded, provided $X$ is complete.
\end{enumerate}

It can be shown that (the full force of) UBP can be derived from (a) whereas a proof of the open mapping theorem can
 be based on (a) and (b) above. Finally, it is well known how to derive closed graph theorem from the open mapping 
theorem (and conversely). For detailed proofs of these statements, one may look into the original source \cite{rh}.

\subsection{Proof of certain classical theorems of analysis}\label{s2.2}
\hspace*{\fill} \\
	As opposed to proofs of classical theorems like the Riemann mapping theorem (in complex analysis) or the 
Muntz-Sasz theorem regarding the description of sequences of positive integers $\{n_{k}\}$ such that the sequence 
of monomials $\{t^{n_{k}}\}$ has a closed linear span in $C[0,1]$ which are indeed quite complicated, it is 
possible to
 provide alternative neat proofs of these statements based on the Hahn Banach theorem. On the other hand, Peter
 Lax showed how to use Hahn- Banach theorem to solve the problem involving the existence of Green’s function for a
 given boundary value problem for the Laplace operator. Details of proofs can be found in \cite{wy} and (\cite[Section 9.2 and Section 9.5]{lax}), respectively.

  \subsection{Invariant version of HB-theorem}\label{s2.4}
 \hspace*{\fill} \\
   The ‘invariant’ version of HBT involves the possibility of extending functionals which are invariant w.r.t a suitable group of transformations to a functional which is also invariant.  More precisely,\\
 Given a group G of transformations acting on a real linear space $X$, a sublinear functional $p$ on $X$, a subspace 
 $Y$ of $X$, and a linear functional $g$ on $Y$, such that
 $$ T(Y)\subset Y,\,g(y)\leq p(y),\,p(Tx)\leq p(x),\,g(Ty)=g(y), $$
 $$\forall x\in X, \,\, \forall y \in Y ~\textnormal{and}~  \forall T\in G.$$
 Then there exists a linear functional $f$ on $X$ extending $g$ such that $f(x)\leq p(x)~  \textnormal{and} ~f(Tx)=f(x), ~\forall x\in X $~and ~$\forall T\in G$.
 
 The above result is well known to follow from the Markov-Kakutani (fixed point) theorem which states that certain 
 groups G of transformations acting as affine maps on a compact convex subset K of a locally convex space have a 
 common fixed point: there exists $x\in K$ such that $T(x)=x$ for all T in G. Among the possible candidates for G, 
 one can choose G to be an abelian group. In fact, the result holds for the more general class of groups called 
 amenable groups. To see how the proof works, assume that $X$ is a normed space with $g\in Y^{\ast},\| 
 g\|=1$ and consider the set
 $$ K=\{h\in X^{\ast};\| h\|\leq 1,h=g~\textnormal{on}~Y\}.$$
 The set $K$ is obviously convex and $weak^{\ast}$-closed. The classical Hahn-Banach theorem shows that $K$ is nonempty and, by virtue of Banach Alaoglu theorem,
  $K$ is $weak^{\ast}$-compact. Further, we see that for each $T\in G$, the assignment
 $$ \Psi_{T}:h\rightarrow h\circ T$$
 defines an affine map of K into itself. For a given $h_{1}\in X^{\ast}$, consider the typical $weak^{\ast}$-neighbourhood of $h_{1}T$:
 $$ U=\{h\in X^{\ast}:\mid h(x_{i})-h_{1}T(x_{i})\mid < \varepsilon\}$$
 determined by (finitely many) $x_{i}$'s in $X$ and $\varepsilon >0 $ and observe that $\Psi_{T}(V)\subset U$ where 
 V is the $weak^{\ast}$-neighbourhood of $h_{1}$ given by:
 $$ V=\{h\in X^{\ast}:\mid h(Tx_{i})-h_{1}T(x_{i})\mid < \varepsilon\}$$
 This shows that $\Psi_T$ is $weak^{\ast}$-continuous and so by the
  Markov-Kakutani theorem stated above, there exists $f\in K$ such that $f\circ T=\Psi_T (f)=f$ for all $T\in G$. This completes the proof.
 \begin{remark}\label{r3}
  As a useful generalisation of Markov-Kakutani theorem to certain noncommutative groups, 
 it is possible to consider a smallest family $\Im$ of semigroups of affine maps on K fulfilling certain conditions
  and then conclude that for each $F \in \Im$, there exists $x\in K$  such that $f(x)=x$, for all $f\in F$. The 
 family $\Im$ is assumed to
  satisfy the following properties:
 \begin{enumerate}
 \item[(i)]    Every abelian group of affine maps of K belongs to $\Im$.
 \item[(ii)]	  If a semigroup F acting on K has a normal subgroup H such that H together with $\frac{F}{H}$ belong to 
 $\Im$, then $F \in \Im$.
 \end{enumerate}
 \end{remark}
  The corresponding version of the ‘invariant’ Hahn-Banach theorem can be stated as follows:
 \begin{theorem}\label{t1}
 Let $Y$ be a subspace of a normed space $X$, $g\in Y^{\ast}, G\subset L(X)$ and let $\Im$ be the smallest family of semigroups of linear maps on $X$ satisfying
 \begin{enumerate}
 \item[(i)] 	Every commutative group of linear maps on $X$ belongs to $\Im$.
 \item[(ii)] If a semigroup F of linear maps on $X$ has a normal subgroup H such that H together with $\frac{F}{H}$ belong to $\Im$, then $\Im$.
 \end{enumerate}
 Further, assume that $G\in\Im$ such that
 \begin{enumerate}
 \item[(iii)] $T(Y)\subset Y,\forall T\in G$.
 \item[(iv)] $g\circ T=g,\forall T\in G$.
 \end{enumerate}
 Then, there exists $f\in X^{\ast}$ such that $f=g~ on~ Y,$ $~\| f\| = \| g\|$~ and~ $f\circ 
 T=f,$ $ \forall\, T \in G.$
 \end{theorem}
 (See \cite{wi} for details of proofs of these statements).
 
 \subsection{Existence of Banach limits}\label{s2.3}
 \hspace*{\fill} \\
 A well known application of the HBT concerns the existence of the so-called Banach limit - a shift invariant continuous linear 
 functional of norm 1 on the space $\ell_{\infty}$ which preserves non-negative sequences and which assigns the limit to each convergent sequence 
 in $\ell_{\infty}$. As a useful consequence of the existence of a Banach limit, it turns out that each compact metric space $X$  admits a countably 
 additive regular Borel measure which is invariant with respect to a given continuous map on $X$. This may be compared with a far deeper assertion
  that each compact metric space $X$ admits a finitely additive regular Borel probability measure which is invariant with respect to the group of
  all isometries on X. The latter statement is a consequence of a certain topological version of the Markov-Kakutani theorem valid for continuous
  G-actions on compact convex subsets of locally convex spaces. Here G is a locally compact group that is also topologically amenable, i.e., G supports a
  finitely additive invariant regular Borel probability measure. Thus every compact group G is topological amenable- the desired measure on G being witnessed
  by the Haar measure on G.
\subsection{Existence of Banach measures on $\mathbb{R}$ and $\mathbb{R}^{2}$}\label{s2.5}
\hspace*{\fill} \\
It is well known that Lebesgue measure on $\mathbb{R}^{n}$ cannot be extended as a countably additive translation invariant measure on all 
subsets of $\mathbb{R}^{n}$. The question arises whether it is possible instead to have a finitely additive rotation invariant measure $\mu$ defined on all
 subsets of $\mathbb{R}^{n}$ which normalises the cube $((\mu([0,1]^{n} )=1)$. (Such measures are called Banach measures in the literature). 
Unfortunately, Banach measures do not exist in $\mathbb{R}^{n}$ at least for $n\geq 3$ (as a result of Banach 
Tarski Paradox in $\mathbb{R}^{n}$ for $ n\geq 3)$. 
However, for n=1, 2, the existence of a Banach measure can be established as a beautiful application of (the invariant form of) Hahn Banach theorem which 
we shall use to prove the existence of a finitely additive, translation invariant measure $\mu$ ~ on~ $P(\mathbb{R}^{n})$~ with ~$\mu(\mathbb{R}^{n} )=1$, (although 
Banach measures are known to exist only in $\mathbb{R}$ and $\mathbb{R}^{2}$). The setting for applicability of the ‘invariant’ 
Hahn Banach theorem is provided by the following data:
$$ G=\mathbb{R}^{n}$$
$$ X=B(G),\,\textnormal{ the ~space~ of~ bounded~ functions~ on ~G},$$
$$ Y=\{f\in B(G);f=c\chi_{G},c\in \mathbb{R}\}$$
$$ p(f)=\sup_{x\in G}\{f(x):f\in G\}$$
$$ G -\textnormal{ action ~on~X~given~by:}f\rightarrow f_{a} \textnormal{~where~} f_{a}(x)=f(a+x)$$
$$ H(g)=c, \textnormal{~for~} g=c\chi_{G}\in Y.$$

Thus, by the invariant Hahn Banach theorem, there exists an F, a linear functional on X satisfying the following properties:
$$ F(g)=H(g),\forall g\in Y$$
$$ F(f)\leq p(f),\forall f\in X$$
$$ F(f)=F(f_{a} ), \forall f\in X,a\in G.$$
Finally, defining $\mu(A)=F(\chi_{A}) \textnormal{~for~} A\subset \mathbb{R}^{n}$, gives a finitely additive translation invariant measure on
 $P(\mathbb{R}^{n}) \textnormal{~ with~} \mu(\mathbb{R}^{n} )=1$.
 \begin{remark}\label{r4}
The same proof works for an arbitrary abelian group in place of $\mathbb{R}^{n}$.
 \end{remark}
 \subsection{Existence of non-measurable sets}\label{s2.6}
 \hspace*{\fill} \\
	The existence of a nonmeasurable set of real numbers is invariably proved as a consequence of the 
axiom of choice which is also used as an important tool in the proof of the Hahn Banach theorem. It is now known 
that Hahn Banach theorem can also be obtained via Ultrafilter theorem which is a consequence of the axiom of 
choice. However, it is also known that none of these implications can be reversed. In particular, Hahn Banach 
theorem does not imply the axiom of choice. In a short but remarkable paper \cite{fn}, Foreman uses the Hahn Banach 
theorem to prove the existence of a nonmeasurable set and therefore, without the use of the axiom of choice$!$ 
 Immediately thereafter, Pawlikowski \cite{pi} used the idea of 
proof in Foreman$'$s theorem to deduce Banach Tarski paradox as a consequence of Hahn- Banach theorem.\\

Before we pass on to the next section where we discuss certain not-so-well-known aspects of the classical HBT, let us briefly pause to mention that there are suitable versions of this theorem which are known to be valid in different settings encountered in analysis. Let us state the following noncommutative version of HBT, due to Wittscock \cite{wk}, which is a fundamental result in the theory of operator spaces with the appropriate morphisms being defined by completely bounded maps.\\

Given a $C^{\ast}$ algebra $\mathcal{A}$, a closed subalgebra $\mathcal{B}\subset \mathcal{A}$ and a completely bounded map $S:\mathcal{B}\rightarrow B(H)$, 
there exists a completely bounded map $T:\mathcal{A}\rightarrow B(H)$ which extends S.\\
\section{Some less known aspects of Hahn-Banach Theorem}\label{sc}
 \begin{enumerate}
 \item[\ref{sc1}]	Operator-analogue of HB-theorem.
 \item[\ref{s3.2}]	Uniqueness issues involving HB-theorem
 \item[\ref{s3.3}]	Bilinear(multilinear) versions of HB-theorem.
 \end{enumerate}
\subsection{Operator-analogue of HB-theorem}\label{sc1}
\hspace*{\fill} \\
There are easy examples to show that the operator-analogue of HB-theorem may fail. Let $Y$ be a nonclosed subspace of a normed space $X$. Then the identity map on $Y$ does not admit a continuous linear extension to $X$.\\
A more nontrivial example where the subspace $Y$ is even closed is provided by the identity map on the closed subspace $c_{0}$  of $\ell_{\infty}$ 
which cannot be extended to a continuous linear map on $\ell_{\infty}$ (R.S.Phillips and A.Sobczyk). In other 
words, we say that $c_{0}$ is not {\it complemented} in $\ell_{\infty}$. (See \cite{sk} for a simple proof).

On the other hand, there are infinite dimensional Banach spaces for which the vector-valued analogue of HB- theorem 
does indeed hold:
\begin{theorem}[R.S.Phillips]\label{t2}
 Given Banach spaces $Z\subset X~and~g:Z\rightarrow\ell_{\infty},\exists 
~f:X\rightarrow\ell_{\infty}~s.t.f|_{Z}=g.$ 
Moreover,  $f$ can be chosen to satisfy \\
$\| g \| = \| f \|$.
\end{theorem}
The above theorem is a straightforward consequence of the Hahn-Banach theorem applied to each component of the given map taking values in $\ell_{\infty}$.

Banach spaces like $\ell_{\infty}$ satisfying the properties of the above theorem have acquired a special name in functional analysis.
\begin{definition}\label{d3}
\textnormal{A Banach space $Z$ is said to be (isometrically) injective  if each $Z$- valued continuous linear mapping $f$ on any 
subspace of a given Banach space can be extended to a continuous linear map $g$ on the whole space while preserving the norm. }
\end{definition}
 If the condition of ‘preserving the norm’ is dropped, it still holds that for some $\lambda > 1 $, one has$\| g\|\leq\lambda\| f\|$ .
 In that case, we say that $Z$ is $\lambda-$injective. Thus injective Banach spaces are exactly those which are 1-injective.

The above result is subsumed as a very special case of the following more general theorem  which completely 
characterizes injective Banach spaces.
\begin{theorem}[Nachbin, Goodner, Hasumi, Kelley, early 1950s]\label{t3}
A Banach space X is $($isometrically$)$ injective if and only if it is
 linearly isometric with $C(K)$ for some compact Hausdorff space K which is extremally disconnected $($i.e., each open set in K has an open closure$)$.
\end{theorem}
 (See \cite[Section 4.3 for a complete proof]{ak}).

For real Banach spaces, the above condition is equivalent to each of the following conditions involving a nonlinear analogue of the Hahn Banach theorem:
\begin{enumerate}
\item[(i)] Every Banach space $Z$ containing $X$ containing $X$ as a subspace admits a norm-one projection which is equal 
to the identity onto $X$.
\item[(ii)] $X$ has the binary intersection property: every family of mutually intersecting closed balls in $X$ has a nonempty intersection (See Section (F)).
\item[(iii)] $\gamma_{X}=1$.(Here, $\gamma_{X}$ denotes Jung’s constant of X which is defined to be the smallest positive number r such that each set of
 diameter 1 in $X$ is contained in a ball of radius $\frac{r}{2}$).
\end{enumerate}
\begin{remark}\label{r5}
\begin{enumerate}
\item[]
\item[(a).] As is easily seen, the sup-norm on $\mathbb{R}^{n}$ (or on $\ell_{\infty}$) obviously has the binary intersection property (because 
of the polyhedral structure of their unit balls)  whereas this property does not hold for the Euclidean norm (even 
on $\mathbb{R}^{2}$). 
This is so because three intersecting circles may have an empty intersection. The same argument applies to provide a proof of the assertion made 
in Proposition \ref{p1} of Section \ref{sa}.
\item[(b).] As seen above, the injective space $\ell_{\infty}$ is nonseparable. The question whether one can choose an injective space to be separable is answered in 
the negative by noting that C(K) is separable exactly when K is metrizable. However, a metrizable space K which is extremally disconnected is necessarily finite,
 in which case C(K) is isometric to $\ell_{n}^{\infty}$. This shows that an injective space is either too big (non-separable) or too small in the sense of being 
finite dimensional. Even more is true: the latter statement is valid even without the assumption of ‘preservation of norms’. This is a consequence of a deep theorem
of Lindenstrauss to the effect that $\ell_{\infty}$ is prime in the sense that all its complemented subspaces are isomorphic to $\ell_{\infty}$. Thus, the assumed 
existence of a separable (not necessarily isometrically) injective space X would entail, by virtue of the separable universality of $\ell_{\infty}$, that X is a
 subspace of  $\ell_{\infty}$ and so by injectivity of X, is complemented in $\ell_{\infty}$ and, therefore, isomorphic to $\ell_{\infty}$ by the primality of
 $\ell_{\infty}$, contradicting the nonseparability of $\ell_{\infty}$.
\item[(c).] It is a famous theorem of Rosenthal \cite{rl2} that an infinite dimensional injective Banach contains (an isometric) copy of $\ell_{\infty}$.
\item[(d).] An injective Banach may not necessarily be a dual space. This was shown by Dixmier, though in a different context, and later also by Rosenthal \cite{rl1}.
\item[(e).] In spite of (b) above, it is still possible to use a separable space Z to effect a HB-extension property for Z-valued mappings, provided the (domain) 
spaces involved are separable. More precisely, we can choose $Z = c_{0}$:
\end{enumerate}
\end{remark}
\begin{theorem}[Sobczyk \cite{sk}]\label{t4}
Given a separable space $Y\subset X~and~g:Z\rightarrow c_{0},\exists 
~f:Y\rightarrow c_{0}~s.~t.f|_{Y}=g.$ 
Further, we can choose $f$ to satisfy: $$\| f\|\leq 2 \| g \|.$$
\noindent{\textnormal{\bf Converse}}\textnormal{ (Zippin \cite{zn1}).} $c_{0}$ is unique (as a separable Banach space) with the above property.
\end{theorem}
\begin{remark}
 Theorem \ref{t4}  holds also for WCG (compactly weakly generated) spaces $X$ as also for those pairs of Banach spaces $(X, Y)$ for 
which $Y$ is a subspace of $X$ such that the quotient space $X/Y$ is separable. Further, there are nonseparable Banach spaces $Z$ for
 which Theorem 5 holds with $c_{0}$ replaced by $Z!$ A recent remarkable development in this area shows that the space $L_{1}^{b}(\mathbb{R})$  
consisting of bounded Lebesgue integrable functions on $\mathbb{R}$ has this property (See \cite{ascgm}).
\end{remark}

The above discussion motivates the following question. Unless otherwise stated, all the mappings appearing in our 
discussion will be bounded linear mappings.\\
Q: Describe the class of Banach spaces $X$ satisfying (HBEP) with respect to $C[0,1]$:\\
(*) For all subspaces $Y\subset X$ and $g:Y\rightarrow C[0,1], \exists f:X\rightarrow C[0,1] s.t. f|_{Z}=g$ 
and $\| f\|=\| g\|$.
\begin{remark}\label{t5}
\begin{enumerate}
\item[]
\item[(a).] It's a highly nontrivial theorem of Milutin (see \cite[Theorem 4.4.8]{ak}) that for all uncountable compact metric 
spaces K, C(K) is 
linearly isomorphic with C[0,1]. As a consequence of Banach-Stone theorem, it follows that these two spaces are not 
always (linearly) isometric. In any case, 
C(K) always embeds isometrically as a subspace of C[0,1]-thanks to Banach-Mazur embedding theorem. In view of these facts, it suffices to check the 
HB-extension property for C(K) spaces by considering the special space C[0,1] as a generic example.\\
\item[(b).] It is not always the case that under the conditions of the above question (Q), HB-extensions can be effected without increase of norm. 
Accordingly, following Kalton, we shall say that the pair $(Y,X)$ of Banach spaces with $Y\subset X$ has $(\lambda,C)$ extension property if 
(*) holds with $\| f\|\leq \lambda \| g\|$. Further, $X$ is said to have $(\lambda,C)$ 
extension property 
if $(Y,X)$ has $(\lambda,C)$ extension property(EP) for all subspaces $Y$ of $X$.
\end{enumerate}
\end{remark}
Some nontrivial examples of this phenomenon are given below.
\begin{theorem}[Lindenstrauss-Pelczynski \cite{lp}]\label{t6}
$c_{0}~has ~(1+\epsilon,C)$-extension property for all $\epsilon >0$.
\end{theorem}
\begin{theorem}[Zippin \cite{zn2}]\label{t7}
For $p\neq 1, \ell_{p}$ has $(1,C)$-extension property.
\end{theorem}
The case p=1 is covered in
\begin{theorem}[Johnson-Zippin \cite{jz}]\label{t8}
The pair $(E,\ell_{1})$ has $(3+\epsilon,C)$-extension property for each $\epsilon>0$ and
 for each subspace E of $\ell_{1}$ which is $w^{\ast}$-closed.
\end{theorem}
\begin{remark}\label{r6}
\begin{enumerate}
\item[]
\item[(a).] A far reaching improvement of Theorem \ref{t8} was proved by Kalton \cite{kn1}, replacing $3+\epsilon ~ by ~1+\epsilon$.
\item[(b).] Theorem \ref{t8} does not hold for norm-closed subspaces of $\ell_{1}.$
\end{enumerate}
\end{remark}
A number of open problems belonging to the domain of operator analogue of Hahn Banach theorem continue to be 
unresolved since they were posed several years ago. We isolate but two of them in view of their relevance to the theme of this paper.
\begin{problem}\label{pb1}
Given a reflexive Banach space $X$ and a subspace $Y$ of $X$, does the pair $(Y,X)$ have $(1,C)$-EP?
\end{problem}
(Compare with Theorem \ref{t7}. The problem is open even for $L_{p}$ spaces, $1<p<\infty$. However, it is known that for finite 
dimensional $Y\subset L_{1}, (Y,L_{1})~ has~ (1,C)-EP$).\\
Following is a celebrated open problem belonging to this circle of ideas.
\begin{problem}\label{pb2}
Is it true that all $\lambda$-injective spaces are isomorphic with a C(K) space for K,
 a compact Hausdorff extremally disconnected space.\\
 \textnormal{ (It is conjectured that the answer is in the affirmative).}
\end{problem}
\subsection{Uniqueness issues involving HB-theorem}\label{s3.2}
\hspace*{\fill} \\
It is natural to ask if there are situations when the HB-extension may be unique. Whereas the answer is easily seen to be in the 
affirmative for Hilbert spaces, a complete characterisation guaranteeing uniqueness of HB-extensions is provided by the following theorem:
\begin{theorem}[\cite{tr}, see also \cite{fl}]\label{tc14}
Given a normed space X, then every continuous linear functional on every subspace 
of X admits a unique norm-preserving HB-extension to $X$ if and only if $X^{\ast}$ is strictly convex.\\
\textnormal{ (A strictly convex space is one whose unit sphere does not contain nontrivial flat segments).}
\end{theorem}
The issue involving characterization of subspaces admitting unique HB-extensions has been addressed in a recent work of P. Bandyopadhyay and A.K. Roy (\cite{br1}) in terms of nested sequences of balls. The ‘invariant’ analogue of this uniqueness result has been treated in a subsequent work \cite{br2} by the same authors.\\
\subsection{Bilinear (multilinear) versions of HB-theorem}\label{s3.3}
\hspace*{\fill} \\
The bilinear version theorem of HB-theorem does not always hold: consider the scalar product on $\ell_{2}$ which cannot be
 extended to a bilinear functional on $\ell_{\infty}$. This is so because bilinear forms on $\ell_{\infty}$ are weakly sequentially 
continuous (because $\ell_{\infty}$ has the Dunford-Pettis property) whereas the inner product is not.  
It is also known that there are closed subspaces $Y$ of $c_{0}$ (and of $\ell_{1}$) and bilinear forms on $Y\times Y$ which cannot be 
extended to bilinear forms on the whole space. The search for Banach spaces $X$ such that for each subspace $Y$ of $X$, every bilinear form
 on $Y\times Y$ can be extended to a bilinear form on $X\times X$, is a celebrated open problem belonging to this circle of ideas. 
Banach spaces $X$ with this property shall be designated to have the bilinear extension property for subspaces (BEPS). 
The dual problem involving the extension of bilinear forms on $X\times X$ to $Z\times Z$ where $Z$ is a superspace of $X$-bilinear 
extension property for containing spaces (BEPC)-is also open despite extensive work that is going on in this area right now. 
The following theorem gives a complete characterisation of bilinear forms on Banach spaces which can be extended to superspaces. 
Here the symbol $\tau_{B}$ denotes the map $\tau_B:X\rightarrow X^{\ast}$ defined by $\tau_B (x)(y)=B(x,y)$ associated to a given 
bilinear form B on $X\times X$.
\begin{theorem}\label{t9}
For a bilinear form B on a Banach space X, TFAE:
\begin{enumerate}
\item[(i)]	B can be extended to a bilinear form on any superspace of $X$.
\item[(ii)]  $\tau_{B}$ is $2$-dominated: $\exists$ a Hilbert space H and\\
  $u\in \Pi_2 (X,H),v\in\Pi_2^d (H,X^* )  \,\,\,  s.t.\tau_B=\nu ou$.
\item[(iii)]  $\exists$ a Hilbert space H and $u,v\in\Pi_2 (X,H)$ such that
 $$B(x,y)=\langle u(x),v(y) \rangle,x,y \in X.$$
\end{enumerate}
\end{theorem}
As a consequence, it follows that for an extendible bilinear form B, the associated map $\tau_B$ is always 2-summing. 
This also provides an alternative argument to the assertion that the scalar product on $\ell_2$ does not extend to some 
superspace of $\ell_2$. On the other hand, there are examples of bilinear forms B such that $\tau_B$ is 2-summing but for 
which B is not extendible. However, such examples cannot occur amongst Hilbert Schmidt spaces as the following simple proposition shows.
(Compare with Theorem \ref{t17a}).
\begin{proposition}\label{p2}
\begin{enumerate}
\item[]
\item[(a)] Let $X$ be a Hilbert Schmidt space. Then each bilinear form B on $X$ for which $\tau_B$ is $2$-summing is extendible.
\item[(b)] A Banach space X having (BEPC) is a Hilbert Schmidt space.
\end{enumerate}
\end{proposition}
Here, by a Hilbert Schmidt space we mean a Banach space $X$ with the property that a bounded linear map acting between 
Hilbert spaces and factoring over $X$ is already a Hilbert-Schmidt map. A useful characterisation of such spaces is given 
by the following theorem which is a consequence of the well known fact that 2-summing maps between Hilbert spaces coincides 
with Hilbert Schmidt maps.
\begin{theorem}\label{t10}
For a Banach space $X$, the following are equivalent:
\begin{enumerate}
\item[(i)]	$X$ is a Hilbert-Schmidt space
\item[(ii)] $L(X,\ell_2)=\Pi_2 (X,\ell_2  ).$
\end{enumerate}
\end{theorem}
\begin{remark}\label{r7}
A Banach space X for which Theorem \ref{t10} (ii) holds with $\Pi_2$ replaced by  $\Pi_1$ is said to satisfy Grothendieck's theorem (or is called a (GT)-space).
\end{remark}
\begin{proof}[Proof of Proposition \textnormal{ \ref{p2}}]
\item[]
(a):  An application of the Pietsch factorisation theorem for 2-summing maps combined with Theorem \ref{t10} 
above and the fact that being Hilbert Schmidt is a self dual property yields the conclusion.\\
(b): Let $u:X\rightarrow H$ be a given operator. By the given hypothesis, the bilinear form B on $X$ given by $B(x,y)=\langle u(x),v(y)\rangle$ is extendible, so, by virtue of Theorem 15, we can write
$$B(x,y)=\langle u_1 (x),v_1 (y) \rangle,u_1,v_1 \in\Pi_2 (X,H),x,y\in X.$$
But then, $u=u_1,v=v_1$ on $X$. Now let, $( x_n )\in \ell_2 [X] $ and observe that
$$\sum_{n=1}^{\infty} \| u(x_n )\|^2 =\sum_{n=1}^{\infty} \langle u(x_n ),u(x_n )\rangle 
=\sum_{n=1}^{\infty}\langle u_1 (x_n ),v_1 (x_n )\rangle$$ $$\leq \Bigg(\sum_{n=1}^{\infty} \| u_1 (x_n ) \|^{2} \Bigg)^{\frac{1}{2}}
 \Bigg(\sum_{(n=1)}^{\infty} \| v_1 (x_n ) \|^{2} \Bigg)^{\frac{1}{2}}<\infty.$$
\end{proof}
For the class of Banach sequence spaces, (BEPC) has been completely characterised as the following recent theorem of Carando and Peris \cite{cs} shows.
\begin{theorem}
 Given a Banach sequence space $X$, then for every pair of  Banach space $Y$ and $Z$ containing $X$, 
every bilinear 
form on $X\times X$ can be extended to a bilinear form on $Y\times Z$ if and only if $X^{\times} = \ell_1.$\\
Here Kothe dual f $X^{\times}$ of f $X$ is the sequence space defined by 
$$X^{\times} = \{ \overline{x} = (x_n) \in \mathbb{C}^{\mathbb{N}}; \sum_{n=1}^{\infty}|x_ny_n|<\infty,\forall(y_n) \in 
X\}.$$
\end{theorem}
\begin{remark}\label{r19}
 Regarding the description of Banach spaces satisfying (BEPS), we observe that
 \begin{enumerate}
\item[(i)] If $X$ is a Hilbert space and $Y$ is a subspace of $X$, 
then every bilinear form on $Y$ can be extended to $X$. This is easily seen by using the orthogonal projection on $Y$. 
Thus Hilbert spaces have (BEPS). More generally we have

\item[(ii)] Hilbert spaces even have the multilinear extension property and, 
in particular, the n- extension property involving the extension of scalar-valued polynomials of degree n.

\item[(iii)] Banach spaces having type 2 enjoy (BEPS). 
This is a consequence of Maurey’s extension theorem(see \cite[Theorem 7.4.4]{ak}) which says that each bounded linear map 
acting on a subspace of a
 Banach space of type 2 and taking values in a cotype 2 Banach space admits a bounded linear extension to the whole space. 
Conversely, Banach spaces having (BEPS) are of weak type 2 and so, are of type p for $1\leq p < 2$.  However, the 
latter property is 
also implied by Banach spaces having 2-extension property. Thus we have the following chain of implications:
type 2$\Rightarrow$(BEPS)$\Rightarrow$ weak type 2$\Rightarrow$ type p for $1\leq p< 2\Leftarrow 2$-ext. property.

 \end{enumerate}
\end{remark}
The above discussion motivates the following questions:
\begin{question}\label{q1}
Is it true that $2$-extension property implies weak type $2$?
\end{question}
\begin{question}
Does (BEPS) imply type $2$?
\end{question}

\section{HB-extension property as an (FD)-property}\label{sd}

We begin with the following question:\\
Do there exist Banach spaces X with the following strengthened HBE- property:\\
$(\ast) \forall$ Banach space $Z$, $\forall$ subspaces $Y\subset X, \forall g\in L(Y,Z),\exists f\in L(X,Z)$  
s.t.$ f|_Y=g$ $\&$ $\forall$ superspaces $W\supseteq X, \forall g\in L(X,Z),\exists f\in L(W,Z)$ s.t. $f|_X=g.$
\begin{remark}
An easy consequence of classical HB-theorem shows that all finite dimensional 
spaces $X$ satisfy $(\ast)$. It is also easy to show that there are no infinite dimensional Banach spaces
 verifying $(\ast)$. Indeed, the latter part of $(\ast)$ yields that X is injective, so by the result of
 Rosenthal quoted in Section \ref{sc}, Remark \ref{r5} (c), X contains a copy of $\ell_\infty$. Now the first part of 
$(\ast)$ gives that $c_0$(as a subspace of $\ell_\infty$ is complemented in X and hence in $\ell_\infty$, which 
contradicts the fact that $c_{0}$ is not complemented in $\ell_\infty$, as was observed in Section \ref{sc1}. 
In other words, $(\ast)$ is a finite-dimensional- property in the following sense. For a detailed account 
of the material included in this section, see \cite{mas1}.
\end{remark}
\begin{definition}\label{d4}
\textnormal{Given a property (P), we say that (P) is a finite-dimensional property ((FD)- property,
 for short) if it holds  for all finite dimensional Banach spaces but fails for each infinite dimensional Banach 
space. }
\end{definition}
\begin{example}
\textnormal{
\begin{enumerate}
\item[]
\item[(i)] Heine-Borel Property( $B_X$  is compact).
\item[(ii)] $ X^\ast=X'$(algebraic dual of $X$).
\item[(iii)] Completeness of the weak-topology on $X$.
\item[(iv)] $\ell_2 \{X\}=\ell_2 [X].$
\item[(v)] (Simultaneous) Hahn-Banach Extension property $(\ast)$.
\end{enumerate}
}
\end{example}
An important (FD)-property is provided by considering the Hilbert Schmidt property encountered earlier in Section \ref{sc}
 (following Proposition \ref{p2}). Some of the important examples of these spaces are provided by
\begin{example}
$c_0,\ell_\infty,\ell_1,C(K),L_\infty (Ω)\cdots.$
\end{example}
\begin{remark}\label{r8}
Hilbert+Hilbert Schmidt = Finite Dimensional i.e., a Banach space is simultaneously Hilbert and of  Hilbert Schmidt type
 if and only if it is finite dimensional.
\end{remark}
\begin{remark}[\textit{ Three important features of (FD)-properties}]\label{r9}
There are three important features involving (FD)-properties but which manifest themselves only in an infinite dimensional context. 
These three features involving a given finite dimensional property (P) derive from:
\begin{enumerate}
\item[(a)] Size of the set of objects failing (P).
\item[(b)] Factorisation property of (P).
\item[(c)] Frechet space analogue of (P).
\end{enumerate}
\end{remark}
{\bf (a).} Given an (FD)- property (P), it turns out that for a given infinite- dimensional Banach space $X$, 
the set of objects in X failing (P) is usually ‘very big’: it could be
topologically big(dense)\\
algebraically big(contains an infinite-dimensional space)\\
big in the sense of category(non-meagre),\\
big in the sense of functional analysis (contains an infinite-dimensional
closed subspace).
\begin{example}\label{e1}
\textnormal{
\begin{enumerate}
\item[(i)]  For an infinite-dimensional Banach space $X$, the set $\ell_2 
[X]/\ell_2\{X\}$ contains a c-dimensional subspace.
\item[(ii)] $ M(X)/M_{b\nu (X)}$ is non-meagre
\item[(iii)] $M([0,1],X)/B([0,1],X)$ contains an infinite-dimensional space.
Here $M(X)$ and $M_{b\nu}(X)$ stand respectively for the space of countably additive $X$-valued measures 
(respectively of bounded variation) whereas $M$ and $B$ 
denote, respectively, the classes of McShane and Bochner integrable functions on $[0,1]$ taking values in X. It 
may be noted that equality of the sets 
displayed in each of these examples is an (FD)-property.
 \end{enumerate}
 }
\end{example}
{\bf (b).} The (FD)-property given in Remark \ref{r8} above comes across as a universal (FD)-property in the following sense:

Given an (FD)-property (P), it turns out that we can write (P) as  a 'sum' of properties (Q) and (R)$:$(P) = (Q) 
$\wedge$ (R) in the sense that a
Banach space $X$ verifies 
\begin{enumerate}
\item[(Q)] iff X is Hilbertian and
\item[(R)] iff X is Hilbert-Schmidt.
\end{enumerate}
In other words, an (FD)-property lends itself to a 'decomposition' (factorisation) as a 'sum' of properties (Q) and (R)
 which are characteristic of Hilbertisability and Hilbert-Schmidt property, respectively.
\begin{example}\label{e2}
\textnormal{ Consider the (FD)-property $\ell_2 \{X\}=\ell_2 [X]$. The factorisation of this (FD)-property entails 
the search for an X-valued sequence space $\lambda(X)$ for which it always holds that $\ell_2 \{X\}\subset 
\lambda(X)\subset\ell_2 [X]$, regardless of $X$ and 
that strengthening these inclusions to equalities results in $X$ being (isomorphically) a Hilbert space and a Hilbert-Schmidt space, 
respectively. In the case under study, the right candidate for $\lambda(X)$ turns out to be the space:
$$ \lambda (X)=\bigg\{( x_n )\subset X: \sum_{n=1}^\infty \| T(x_n)\|^2 <\infty,\forall ~T \in L(X,\ell_2)\bigg\}$$
for which it holds that $\ell_2 \{X\} \subset\lambda (X) \subset \ell_2 [X]$. Now, the desired factorization of 
the equality in Example \ref{e1} (i)
 is given in the following theorem:}
\end{example}
\begin{theorem}
 For a Banach space X, the following statements are true:
 \begin{enumerate}
\item[(a)] $\ell_2 \{X\} =\lambda(X)$ if and only if X is Hilbertian.
\item[(b)] $\lambda(X)=\ell_2 [X]$ if and only if X is a Hilbert-Schmidt space.
\end{enumerate}
\end{theorem}
\begin{example}\label{e3}
\textnormal{The (FD)-property involving the (strengthened) Hahn-Banach property$(\ast)$ (for $Z=\ell_2$) admits a beautiful factorisation as is illustrated 
in the following theorems. }
\end{example}
\begin{theorem}\label{t11}
For a given Banach space $X$, TFAE:
\begin{enumerate}
\item[(i)] A bounded linear map on $X$ into  $\ell_2$ extends to a bounded linear map on any
 superspace of X.
\item[(ii)] $X$ is a Hilbert-Schmidt space.
\end{enumerate}
\end{theorem}
Proof of this equivalence is an easy consequence of Theorem \ref{t10} combined with the following well known facts involving 2-summing maps and 
C(K) spaces:
\begin{enumerate}
\item[(a)]	(Grothedieck’s little theorem): A bounded linear map from a C(K) space into $\ell_2$ is always 2-summing.
\item[(b)]	A 2-summing map defined on a subspace of a Banach space can always be extended to a 2-summing map ( in particular a bounded linear map)
\item[(c)]	Every Banach space can be (isometrically) embedded into a C(K) space.
\end{enumerate}
\begin{theorem}\label{t12}
For a Banach space $X$ with $dim~~X>2$, TFAE:
\begin{enumerate}
\item[(i)] A bounded linear map defined on a closed subspace of $X$ and taking values in an arbitrary Banach space $Z$ extends to a bounded linear map on $X$.
\item[(ii)] $X$ is a Hilbert space.
\end{enumerate}
\end{theorem}
Proof shall be taken up in Section \ref{se}.
\begin{remark}
(i). The isometric analogue of the above theorem is an old result of Kakutani \cite{ki}  which asserts
 the equivalence of condition (i) of Theorem \ref{t12} subject to preservation of norm of the extended map to X being isometrically a Hilbert space.\\
\textit{Remark} 4.13. (ii). It is worth remarking that whereas in Theorem \ref{t11} the single space $\ell_2$ has been used as the range space for mappings 
to characterize $X$ as a Hilbert Schmidt space, that is not the case with Theorem \ref{t12} where $X$ has been characterized as a Hilbert
 space by allowing the range space of mappings to vary over the whole collection of Banach spaces. As has already been seen, using 
the whole lot of Banach spaces in place of the single space $\ell_2$ in Theorem \ref{t11} would lead to $X$ being an isometric copy of C(K) 
for a compact Hausdorff exremally disconnected space K. On the other hand, allowing the single space $\ell_2$ in Theorem \ref{t12} instead of 
the entire class of Banach spaces for mappings on $X$ to take their values in, would result in $X$ being merely a weak type 2 Banach space.
 The latter statement is a partial converse of the famous Maurey's extension theorem alluded to in Section \ref{sc}  (Remark \ref{r19}), due to P. Casazza 
and N. J. Nielsen \cite{cn}. As noted by these authors, the result follows immediately from a theorem of V. D. Milman and G. Pisier   
  ( Banach spaces with a weak cotype 2 property, Israel J. Math. 54(1986), 139-158).
\end{remark}
\begin{theorem}\label{t14}
Let $X$ be Banach space such that each $\ell_2$- valued bounded linear maps on subspaces of $X$ 
can be extended to a 
 bounded linear map on the whole space $X$. Then $X$ is of weak type $2$. In particular, $X$ has  type p for 
$1\leq p<2$.
\end{theorem}
The above discussion motivates the following question.
\begin{question}\label{q15}
Is it possible to enlarge the class of maps on $X$ (and taking values in $\ell_2$) appearing in Theorem \textnormal{\ref{t11}} and Theorem \textnormal{\ref{t14}} 
in order to characterise injectivity and Hilbertisability, respectively, of the Banach space $X$?
\end{question}
It turns out that the question has an affirmative answer if extension is sought of the (considerably larger class of) Lipschitz maps in place of 
bounded linear maps on subsets to supersets of the domain space.(See Theorem \ref{te7} and Theorem \ref{te8}).
\begin{question}\label{q16}
 Do there exist infinite dimensional Banach spaces X such that the following holds:
$(\ast\ast)~ \forall$~ subspaces~ $Y\subset X,$~ $\forall g\in L(Y,\ell_2 ),$ $\exists f\in L(X,\ell_2 )~$  
s.t.~ $f |_{Y}=g~ \forall~$ superspaces $~W\supseteq X, \forall ~g∈L(X,\ell_2 ),\exists ~f\in L(W,\ell_2 )$ ~ 
s.t.~$f |_{X}=g.$
\end{question}
{\bf Answer: NO!}\\
Unlike the proof in the negative of the question posed in $(\ast)$, the proof of the above statement is much deeper and draws upon 
several techniques involving local theory of Banach spaces, including Theorem \ref{t14}. Indeed, the first part of $(\ast\ast)$ 
combined with the indicated theorem above yields that X has weak type 2 and so has nontrivial type. However, Theorem \ref{t11} of Section \ref{sd} 
applied to the first part of $(\ast\ast)$ gives that X is a Hilbert Schmidt space. Finally, an
application of Dvoretzky’s spherical sections theorem yields that an infinite dimensional Hilbert Schmidt space cannot have nontrivial 
type!(See \cite[Chapter 19,Notes and Remarks]{djt})

We now provide bilinear analogues of Theorems \ref{t11} and \ref{t12} which admit the following formulations:
\begin{theorem}\label{t17a}
For a Banach space $X$, TFAE:
\begin{enumerate}
\item[(i)] For each $L_\infty$ space $Z$, $X\times Z$ has (BEPC).
\item[(ii)] $X^\ast$ has (GT).
\end{enumerate}
\end{theorem}
\begin{theorem}\label{t17b}
For a Banach space $X$, TFAE:
\begin{enumerate}
\item[(i)] For each closed subspace $Z$ of $X$, each bilinear form on $Z\times Z^\ast$ admits an extension to a bilinear 
form on $X\times Z^\ast$.
\item[(ii)] X is (isomorphically) a Hilbert space.
\end{enumerate}
\end{theorem}
The above theorem is a recent result of the author \cite{mas2} which is the isomorphic analogue of an old result of T.L.Hayden \cite{h}. The proof is postponed to Section \ref{se}.

The next question deals with the ‘single space’ analogue of $(\ast)$. We state it as a problem.
\begin{problem}\label{p18}
Given a Banach space $X$ such that for all subspaces $Y\subset X$ and $\forall g\in L(Y,\ell_2 ),$  $\exists ~f\in 
L(X,\ell_2 )$~  s.t.~$f|_{Y}=g$. 
Does it follow that $X$ has type $2$? Does the assertion follow at least under the stronger hypothesis that the stated condition is assumed to hold for all 
cotype $2$ Banach spaces in place of $\ell_2$?
\end{problem}
{\bf(c).} When suitably formulated in the setting of Frechet spaces $X$, it turns out that in most of the cases, there exist 
infinite dimensional Frechet spaces in which it is possible to salvage a given (FD)- property. It also turns out that, at least in most cases 
of interest, the class of Frechet spaces in which this holds coincides with the class of nuclear spaces in the sense of Grothendieck. 
The fact that a Banach space can never be nuclear unless it is finite dimensional shows that the nuclear analogue of an (FD)-property provides a
 maximal infinite dimensional setting in which the given (FD)-property holds. Rather than list a whole lot of examples of (FD)-properties which 
characterize nuclearity in Frechet spaces, we shall settle for the following example involving the Frechet space analogue of the Dvoretzky-Rogers 
property which holds exactly when the underlying Frechet space is nuclear.
\begin{problem}\label{p19}
 For a Frechet space $X$, TFAE:
 \begin{enumerate}
\item[ (a)]	$\ell_2 \{X\}=\ell_2 [X]$
\item[ (b)]	$X$ is nuclear.
 \end{enumerate}
\end{problem}
Continuing along these lines, we conclude with the following problem involving the Frechet analogue of the (strengthened) HBEP $(\ast\ast)$:
\begin{problem}\label{p20}
 Given a Frechet space $X$, does it follow that $X$ is nuclear if the following holds:\\
$(\spadesuit)$ $\forall$ ~ subspaces ~$ Y\subset X,~ \forall g\in L(Y,\ell_2),~\exists~ f\in L(X,\ell_2)$ ~ 
s.t.~$f|_{Y}=g ~\& ~\forall$~ superspaces~ $W\supseteq X, ~\forall g\in L(X,\ell_2),~\exists ~f\in L(W,\ell_2)$~  
s.t.~ $f|_{X}=g.$
\end{problem}

The problem is open even under the stronger hypothesis of $(\ast)$ with the target space $\ell_2$ being replaced by the class of all Banach spaces. 
Let us also note in passing that $(\spadesuit)$ holds for all nuclear spaces.

\section{Hilbert spaces determined via Hahn Banach phenomena}\label{se}
We have seen that Hilbert spaces arise as the natural class of Banach spaces for which certain variants of the Hahn Banach property hold 
(Theorem \ref{t12} and Theorem \ref{t17b}). There are many more examples of such phenomena involving the Hahn Banach property which can be used to
 characterise Hilbert spaces. Before we proceed to discuss these characterisations, we shall pause to provide proofs of the above quoted theorems
 as promised in Section \ref{sd}.
\begin{proof}[Proof of Theorem \textnormal{\ref{t12}}]
Whereas (ii)$\Rightarrow$(i) is a straightforward consequence of the projection 
theorem in Hilbert space, the proof of (i)$\Rightarrow$(ii) is much deeper and depends upon the famous Lindenstrauss-Tzafiriri 
theorem (see \cite[Chapter 12]{ak}) characterising Hilbert spaces isomorphically as those Banach spaces in which the projection 
theorem holds for closed subspaces: for each closed subspace M of $X$, there exists a continuous linear projection 
of $X$ onto M.
 It is easily seen that this latter property is a rewording of the statement in (i). The crucial component of the
 equivalence of this latter property with $X$ being a Hilbert space involves  the proof of the assertion that it 
is actually a ‘local property’ 
in the following sense.
\end{proof}
\begin{lemma}\label{5.1}
For a Banach space $X$, TFAE:
\begin{enumerate}
\item[(i)]	For each closed subspace M of $X$, there exists a continuous projection of $X$ onto M.
\item[(ii)] There exists $c>1$ such that for each finite dimensional subspace M of $X$, there
 exits a continuous projection P of $X$ onto M such that $\| P\| \leq c$.
\end{enumerate}
\end{lemma}
It then turns out that there exists $f(c)>0$ such that for each finite dimensional subspace M of X, $d(M,\ell_2^{\dim M})\leq f(c)$, where $d(.,.)$ 
denotes the Banach Mazur distance. By \cite{j}, this holds precisely when X is (isomorphic to a) Hilbert space.
\begin{proof}[Proof of Theorem \textnormal{\ref{t17b}}]
 As in the case of Theorem \ref{t12}, the implication (ii) $\Rightarrow$ (i) is a 
consequence of the existence of projections on closed 
subspaces of a Hilbert space. To show that (i)  $\Rightarrow$ (ii), let $Y$ be a closed subspace of $X$ and 
consider the bilinear form b on  $Z\times Z^\ast$ 
given by: $b(y,y^\ast)= y^\ast (y)$.  Let $\gamma$ be an extended bilinear form on $X\times Z^\ast$ guaranteed by (i). By the Aron-Berner extension 
theorem \cite{abe}, $\gamma$ has a further extension to a bilinear form on  $X^{\ast\ast}\times Z^{\ast\ast\ast}$. Denote 
by $\pi$, the restriction of this bilinear
form on $X^{\ast\ast}\times Z^{\ast}$ and consider the map $P: X^{\ast\ast} \rightarrow Z^{\ast\ast}$ defined by
$$ P(x^{\ast\ast})(y^{\ast})=\pi(x^{\ast\ast},y^{\ast}), y^{\ast} \in Z^{\ast}.$$
It is clear that P is continuous and linear. To see that P is actually a projection, let  $x^{\ast\ast}\in Z^{\ast\ast}$ and fix $y^{\ast} \in Z^{\ast}$.  
An application of Davie-Gamelin theorem \cite{dg} applied to the pair $(x^{\ast\ast},y^{\ast})$ yields that there exists 
 a net $(x_\alpha)\subset Z$ such that $(x_\alpha)$ converges weakly
 $to~ x^{\ast\ast}~ in~ Z^{\ast\ast}~ \text{and}~ \pi(x_\alpha,y^\ast)\rightarrow \pi(x^{\ast\ast},y^\ast).$ This gives
$$ P(x^{\ast\ast})(y^\ast)=\pi(x^{\ast\ast},y^\ast)=\lim\pi(x_\alpha,y^\ast)=\lim ⁡b(x_\alpha,y^\ast)=\lim y^\ast (x_\alpha )=x^{\ast\ast} (y^\ast).$$
Thus $P(x^{\ast\ast})=x^{\ast\ast}$, which means that P is a projection. Finally, let $Q: X^{\ast\ast\ast}\rightarrow X^\ast$ be 
the Dixmier projection defined by $Q(x^{\ast\ast\ast})(x)=x^{\ast\ast\ast}(\hat{x})$ and compose it with the 
adjoint map
 $P^\ast: Z^{\ast\ast\ast}\rightarrow X^{\ast\ast\ast}$ to get a map from  $Z^{\ast\ast\ast}$ onto  $X^\ast$ and let $\psi: Z^\ast\rightarrow X^\ast$ 
be the restriction of this map to $Z^\ast$. It is easily seen that $\psi$ is continuous and linear. An application 
of \cite[Theorem 2.1]{fu} then completes the proof. 
(See also Theorem \ref{te5}).
\end{proof}
Our next Hilbert space characterisation is motivated by the set of equivalent conditions ensuring the uniqueness of the Hahn Banach extension.
 As was noted in Theorem \ref{tc14}, this happens exactly when the dual space is strictly convex which can also be shown to be equivalent to the so called 
2-ball sequence property defined below.
\begin{definition}\label{ed1}
\textnormal{Given a Banach space $X$ and $k\in \mathbb{N}$, a subspace $Y$ of $X$ is said to have the {\it $k$-ball 
sequence property} if for all sequences
 of balls $\bigg\{B_n^{(i)}\bigg\}_{(i=1)}^k$ in $X$, $B_{n}^{(i)}=B(x_n^i,r_n^i)$ such that $B_n^{(i)} \subset 
B_{n+1}^{(i)},r_n^i\rightarrow\infty,x_{n+1}^i-x_n^i\in Y,~\forall ~i=1,2,\ldots, k$  and 
$$  \bigcap_{i=1}^k\bigcup_{n=1}^\infty B_n^{(i)}\neq\phi,$$
it follows that 
$$ Y\bigcap\Big(\bigcap_{i=1}^k\bigcup_{n=1}^\infty B_n^{(i)}\Big)\neq\phi.$$
We shall say that X has the k-ball sequence property if each (closed) subspace of $X$ has it.
}
\end{definition}
\begin{remark}\label{er1}
It is easily seen that k+1-ball sequence property implies k-ball sequence property. More importantly, 3-ball sequence property 
is equivalent to k-ball sequence property for all $k>3$. Furthermore, there are examples of subspaces with 2-ball sequence property but lacking
the 3-ball sequence property.
\end{remark}
 As a strengthening of the uniqueness property of the HB-extensions ({\it U-property}, for short), we introduce the notion of Banach spaces having the 
SU-property (strong uniqueness property) defined below.
\begin{definition}\label{ed3}
 A Banach space $X$ is said to have the SU-property if it has the U-property and each closed 
subspace $Y$ of $X$ is an ideal in $X$, i.e.,
 $Y^\bot$ (the annihilator of $Y$ in $X^\ast$) = ker(P) for some norm-one projection P in $X^\ast$.

\end{definition}
\begin{theorem}[\cite{op}]\label{te4}
For a Banach space $X$, TFAE:
\begin{enumerate}
\item[(i)]	$X$ has SU.
\item[(ii)] $X$ has $3$-ball sequence property.
\item[(iii)] $X$ has k-ball sequence property for all $k>3$.
\item[(iv)] $X$ is a Hilbert space.
\end{enumerate}
\end{theorem}
The next result addresses the issue of selecting a Hahn Banach extension in a manner which gives rise to a linear and continuous 
(extension) operator from $Y^\ast$ into $X^\ast$ for each (closed) subspace $Y$ of $X$. More precisely, we have the following theorem which was used in the proof of Theorem \ref{t17b}.
\begin{theorem}\label{te5}
For a Banach space $X$, TFAE:
\begin{enumerate}
\item[(i)]	For each subspace $Y$ of $X$, there exists a continuous linear map $\psi: Y^\ast \rightarrow X^\ast$ such 
that $\psi(g)|_{Y}=g~ for~ each~ g\in Y^\ast.$
\item[(ii)] X is a Hilbert space.
\end{enumerate}
\end{theorem}
\begin{remark}\label{er6}
Interestingly enough, Theorem \ref{te5} also has a polynomial version which asserts that the above conditions are also equivalent 
to the existence of a continuous linear (extension) operator from $P^n (Y)$  into $P^n (X),n>1$, where the symbols denote the spaces of 
(scalar-valued) polynomials on $Y$ and $X$, respectively. Now the proof of 
(i) $\Rightarrow$ (ii) depends upon the fact that the 
condition (i) of Theorem \ref{te5} is a local 
property in the following sense:

 $(\ast)$There exists $c>0$ such that for each finite dimensional subspace $Y$ of $X$, there exists a continuous 
linear extension map
 $\psi: Y^\ast \rightarrow X^\ast $~ with~$ \| \psi\|\leq c$ (See \cite{fu} for details).
\end{remark}
Indeed, consider finite dimensional  subspaces $Y\subset Z\subset X$ and let $\psi: Y^\ast\rightarrow 
X^\ast$ be a continuous linear 
extension map with $\| \psi\|\leq c$. Taking conjugates gives a map  $\psi^\ast: X^{\ast\ast}\rightarrow Y^{\ast\ast}$ 
which when restricted to $Z$ produces a projection $P:Z\rightarrow Y$ with
 $$ \| P\|\leq\|\psi^\ast\|=\|\psi\|\leq c $$

From the  discussion following Lemma \ref{5.1}, it follows that  there exists $f(c) > 0,$ independent of $Z\subset X 
$ such that  $d(Z,\ell_2^{\dim Z})\leq f(c)$. Since $Z$ was chosen arbitrarily, we conclude that  $X$ is 
isomorphic to a Hilbert space .\\
 Before we discuss our final result involving a characterisation of Hilbert spaces in terms of Hahn Banach extensions, let us recall Question \ref{q15} and
 the comments following immediately thereafter, asserting the possibility of such a characterisation in terms of extendibility of Lipschitz maps. 
We conclude by briefly indicating the ingredients of a proof of (a special case of) the following statement which says that in the case that the 
target space is chosen to be $\ell_2$, Lipschitz maps on arbitrary subsets can be extended to a Lipchitz map on 
the whole space precisely when the underlying space is Hilbertian.
\begin{theorem}\label{te7}
For a Banach space $X$, TFAE:
\begin{enumerate}
\item[(i)]	Given a subset A of $X$ and a Lipschitz map $f:A\rightarrow \ell_2$, there exists a Lipschitz map 
$g:X\rightarrow\ell_2$ such that $f=g$ on A.
\item[(ii)] $X$ is a Hilbert space.
\end{enumerate}
\end{theorem}
\begin{proof}[Proof of Theorem \textnormal{\ref{te7} (outline)}]
(a) The fact that (ii) implies (i) is (the infinite dimensional analogue of) an old result of Kirszbraum \cite{km} which 
says that $\mathbb{R}^m$-valued c-Lipschitz
 maps acting on subsets of Euclidean spaces can always be extended to a c-Lipschitz map on the whole space. The 
proof is straightforward for functions taking values in $\mathbb{R}$.
 Indeed, if A is a subset of $\mathbb{R}^n$ and $f:A\rightarrow \mathbb{R}$ is a c-Lipschitz map, then the map
$$ F(x)=\inf \{f(a)+c\mid x-a\mid:a\in A\}$$
defines a c-Lipschitz map on $\mathbb{R}^n$ which extends $f$. (It is clear that the same argument works for a metric space in place of $\mathbb{R}^n$- an important 
fact which 
is due originally to McShane). Further, the special case of Kirszbraum's theorem proved above applied to the 
co-ordinate functions gives a cL-Lipschitz extension
 of a given c-Lipschitz function $f:A\rightarrow \mathbb{R}^m (A\subset \mathbb{R}^n)$ where L is a constant depending upon m. (Precisely, $L=\sqrt{m}$). In effect,
 Kirszbraum's
original theorem is exactly the last statement modulo the assertion that the constant L appearing there is redundant! However, that is much harder to prove. Here, 
let us also point out that it is relatively easier to prove Kirszbraum's theorem for the special case involving 
1-Lipschitz mappings defined on convex domains in a 
Hilbert spaces and taking values in an arbitrary Banach space. The key idea involved in the proof of this statement is the existence of a Chebyscheff projection 
induced by a convex subset of a Hilbert space where it also turns out to be a contraction, i.e. a 1-Lipschitz map. (See Section \ref{sf} for full details of the proof 
of Kirszbraum's theorem in the infinite dimensional Hilbert space setting).\\
(b) Regarding the implication (i)$\Rightarrow$ (ii), it turns out that a stronger statement is true, namely that 
if (i) holds for some Banach space Z in place of $\ell_2$ which
 is merely assumed to be strictly convex, then both $X$ and $Z$ are Hilbert spaces! On the other hand, (ii) can be shown to hold even for arbitrary Banach space $Z$ in place of $\ell_2$ as long as $X$ is assumed to be strictly convex.\\
(c) The proof of (b) depends upon the following well known geometrical inequality which is due originally to Nordlander. (See \cite{bl}, Proposition \ref{p1} for a proof).

$(\ast)$ Given a normed space $X$ with $\dim X \geq 2 ~\text{and}~ 0<t<1$, then
$$ \inf\{\| x+y\|_{X};(x,y)\in S_{t}\}\leq 2\sqrt{1-t^2}\leq\sup\{\| x+y\|_{X};(x,y)\in S_{t}\}$$ where
$$ S_{t}=\{(x,y)\in S_X\times S_X:\| x+y\|_{X}=2t\}$$
(The proof is a nice application of Green's theorem of advanced calculus).\\
(d)	The conclusion in (ii) is achieved under the weaker assumption- which trivially follows from (i) - 
that 1-Lipchitz maps defined on three-point subsets of the domain can be extended to a 1-Lipschiz map on any four-point superset.
 Here, it is important to point out that the extension is always possible if it is sought to be effected from a 2-point set to a
 larger set having three points. Indeed, if $ A= \{x,y\}$  and if $u: A\rightarrow Y$ is a 1-Lipschitz map, then 
for $z\in X$ and for t defined by
$$ t=\min\Bigg\{1,\frac{\| z-y\|}{\| x-y\|}\Bigg\}$$
the formula $u(z)=tu(x)+(1-t)u(y)$ gives a 1-Lipschitz mapping from $\{x,y,z\}$ into $Y$, which obviously extends the given u.
\end{proof}
\begin{remark}
As seen in the proof of $(i)\Rightarrow (ii) $ above, the extendibility of a Lipschitz map from arbitrary subsets to the whole space $X$ can be equivalently described in terms of the extendibility of the given map from finite subsets of $X$ to the whole space $X.$ This statement may be looked upon as the non-linear analogue of Maurey's extension property (referred to in Remark \ref{r19} (iii)) as a `local' extension property in term of `uniform' extendibility of linear maps on finite dimensional subspaces of a Banach space and taking values in $\ell_2.$ 
\end{remark}
An analogous result involving a nonlinear analogue of Theorem \ref{t3} (i) is the following (see \cite[Propositions 1.2 and 1.4]{bs}).
\begin{theorem}\label{te8}
For a Banach space $X$, TFAE:
\begin{enumerate}
\item[(i)]	For every pair of metric space $Y$ and $Z$ s.t. $X\subset Z$, each $1$-Lipschitz map on $X$ into $Y$  can 
be extended to a $1$-Lipschitz map from $Z$ into $Y$.
\item[(ii)] $X$ is a $1$-absolute Lipschitz retract, i.e., every Banach space containing $X$ as a subspace admits a $1$-Lipschitz map which is equal to the identity map on $X$.
\item[(iii)] $X$ is $1$-injective.
\end{enumerate}
\end{theorem}
\begin{remark}\label{er9}
(a). The above theorem does not extend to 2-injective spaces. This is because $c_0$ is known to be a 
2-absolute Lipschitz retract whereas it is not 2-injective, as Theorem \ref{t4} informs us: it is only 
separably 2-injective, being characterised by the Hahn Banach extension property within the class of all separable 
Banach spaces as opposed to the whole class of Banach spaces suggested by 2-injectivity. Another example is 
provided by the space C(K) for a compact metric space. The fact that it is 2-injective was proved relatively 
recently by Kalton \cite{kn3} as a far reaching improvement of an old result of Lindenstrauss \cite{ls2} who had shown this 
space to be 20-injective.\\
(b). It is natural to ask if the Lipschitz extension property (LEP) as guaranteed by Theorem \ref{te7} extends to pairs of 
Banach spaces $(X,Y)$ where $Y$ may be chosen to be a Banach space other than $\ell_2$. It turns out that the 
choice of $Y$ as a strictly convex Banach space doesn't enlarge the class of Banach spaces $X$ so that $(X,Y)$ 
would have the LEP- such an $X$ would necessarily be a Hilbert space (and in that case, $Y$ is also a Hilbert 
space). However, it is possible that a pair $(X,Y)$ has LEP in which both $X$ and $Y$ are non-Hilbertian Banach 
spaces. This can be seen by considering, for instance, an arbitrary Banach space $X$ and applying McShane's 
theorem (referred to in the proof of Theorem \ref{te7} above) to each component of a $\ell_\infty$- valued map on $X$ to 
conclude that the pair $(X,\ell_\infty)$ has LEP. Whether in this statement,$\ell_\infty$ could be replaced by any 
other infinite dimensional Banach space is a highly nontrivial question which was settled partially by Johnson, 
Lindenstrauss and Schechtman \cite{jls}. They proved  that the pair $(X,Y)$ has LEP as long as $X$ is an 
n-dimensional space and Y is an arbitrary Banach space. Here it is to be noted that the Lipschitz constants L and 
$L'$ of the maps involved satisfy: $L'\leq nL$.\\
(c). In sharp contrast to the finite dimensional Banach spaces where, by virtue of Brouwer's fixed point theorem, 
the unit sphere is not a retract of the unit ball, in the infinite dimensional case there even exists a Lipschitz 
extension of the identity map on the unit sphere to the unit ball. This is a celebrated theorem of Benjamini and 
Sternfeld  \cite{bs}.\\
(d). As noted in Remark \ref{r2} (ii), referring to the Hahn Banach extension theorem as a typically locally convex phenomenon, a nonlinear analogue of that statement in terms of Lipschitz maps is also valid, thus leading to local convexity of a quasi Banach space X as being equivalent to the pair $(X,R)$ having LEP. This is easily seen to be the case by noting that the formula
$$ \| x\|_{0}=\sup\{f(x):f\in \text{Lip}_{0}(X),\| f\|_{\text{Lip}}\leq 1\}$$
defines a seminorm on (the quasi Banach space) $(X, \| \|)$. Here $\| f\|_{\textnormal{Lip}}$ , the 
Lipschitz norm of a Lipschitz function $f:X\rightarrow {\mathbb{ R}},  (f(0)=0)~(f\in \textnormal{Lip}_{0}(X))$ is defined by
$$ \| f\|_{\textnormal{Lip}}=\sup\Bigg\{\frac{\mid f(x)-f(y)\mid}{\mid x-y\mid};x,y\in X,x\neq y\Bigg\}$$
Now, for a given $x\in X, \| x\|=1$, consider the linear functional $f:Z\rightarrow R , tx\rightarrow t$ on the subspace
 $Z$ generated by $x$ and let F be an L-Lipschitz extension of $f$ on $X$. By definition of the quantities defined above, we have
$$ \| x\|_{0}\geq \frac{\mid F(x)\mid}{L}=\frac{f(x)}{L}=\frac{1}{L}> 0$$
and deduce that $\| \|_0$ is actually a norm on X. Finally, the estimates
$$ \| x\|_{0}\leq \| x\|\text{~and}~\| x\|\leq L \| x\|_{0},~\text{for~all}~x\in X$$
together yield that the given quasi norm is equivalent to the norm $\| \|_{0}$.

(e)  A metric analogue of the above theorem (Theorem \ref{te8})- with $X$ being a metric space- is also valid and can 
be described in terms of a certain intersection property of balls of the underlying metric space called 
{\it ‘hyperconvexity’}.  This is discussed in the next section. 
\end{remark}
\section{Intersection of balls and extendibility of maps}\label{sf}

We have already come across situations involving the role played by the intersection property of balls in ensuring extendability of bounded linear maps.
 This was seen to be the case in (the discussion following) Theorem \ref{t4} where injectivity of Banach spaces was characterised in terms of
 the binary intersection property.  On the other hand, Theorem \ref{te4} informs us that Hibertisability of $X$ can be described in terms of
 3-ball sequence property of $X$ which was defined in the same section (Definition \ref{ed1}). It is interesting that uniqueness property of Hahn Banach 
extensions can also be expressed in terms of intersection property of balls (See the comments preceding Definition \ref{ed1}).
\begin{theorem}\label{tf1}
For a Banach space X, TFAE:
\begin{enumerate}
\item[(i)]	X has $2$-ball sequence property.
\item[(ii)] X has the uniqueness property of Hahn Banach extensions.
\end{enumerate}
\end{theorem}
\begin{remark}\label{fr1}
Given a Banach space $X$, a Banach space $Z$ containing $X$and a bounded linear map $g$ 
on $X$, it is formally easier to look for a bounded linear extension of $g$ to a larger space within $Z$ than to 
the whole space $Z$. The description of such phenomena involving extension to superspaces of $X$ in which $X$ is 
1-dimensional leads to a certain ball-intersection property which is the cardinal analogue of the binary 
intersection property(BIP) encountered earlier.
\end{remark}
\begin{definition}\label{fd3}
\textnormal{  For a given cardinal number $\gamma$, we say that a Banach space $X$ has the $\gamma $- 
{\it intersection property} if each family of $\gamma$-many mutually intersecting closed balls in $X$ has a nonempty 
intersection.}
\end{definition}
We make the following observations:
\begin{enumerate}
\item[(a)](BIP) is the same as $\gamma -$ intersection property  for all cardinal numbers $\gamma$.
\item[(b)] $\gamma_{1}$- intersection property  implies $\gamma_2$- intersection property  for all $\gamma_1\geq 
\gamma_2$.
\item[(c)] 4-intersection property is equivalent to k-intersection property for all $k\geq 4$.
\end{enumerate}
The following theorem of Lindenstrauss \cite{ls1} describes the situation where it is possible to extend bounded linear maps from a 
Banach space to a larger space in which it is 1-codimensional.
\begin{theorem}\label{tf4}
For a Banach space $X$, TFAE:
\begin{enumerate}
\item[(i)]	For every pair of Banach spaces $Z$ and $Y$ with $Z\supset X, ~ \dim \frac{Z}{X}=1,~ \dim Y= 2$, every operator from $X$ into $Y$ has a norm-preserving extension from $Z$ into $Y$.
\item[(ii)] $X$ has $3$-intersection property.
\end{enumerate}
\end{theorem}
A stronger result involving the extension of operators taking values in a larger space is the following famous theorem of Lindenstrauss \cite{ls1}.
\begin{theorem}\label{tf5}
For a Banach space $X$, TFAE:
\begin{enumerate}
\item[(i)]	For every pair of Banach spaces $Z$ and $Y$ with $Z \supset X, ~\dim \frac{Z}{X}=1,~ \dim Y= 3$, every operator from X into Y has a norm-preserving extension from $Z$ into $Y$.
\item[(ii)] For every pair of Banach spaces $Z$ and $Y$ with $Z \supset X$, every compact operator from $X$ into $Y$ has a compact norm-preserving extension from $Z$ into $Y$.
\item[(iii)] For every pair of Banach spaces $Z$ and $Y$ with $Z \supset X$, every weakly compact operator from $X$ into $Y$ has a weakly compact norm-preserving extension from $Z$ into $Y$.
\item[(iv)] $X^\ast$ is an $L_1 (\mu)$ space for some measure $\mu$.
\item[(v)] $X$ has $4$-intersection property.
\end{enumerate}
\end{theorem}
It is still desirable to extend bounded linear operators from subspaces $Y$ of $X$ to (larger) subspaces of $X$ including $Y$ as a 
1-codimensional subspace, but taking values in an arbitrary Banach space. This is achieved in terms of the so called $ \gamma$- ball 
property defined as follows.
\begin{definition}\label{fd6}
\textnormal{ Given a subspace $Y$ of $X$ and a cardinal $\gamma$, we say that $Y$ has the $\gamma$-ball property if each family of $\gamma$- many closed 
balls intersecting in $X$ and having their centres in $Y$ also intersect in $Y$.
 }
\end{definition}
We note the following features of this property:
\begin{enumerate}
\item[(a)] k+1-ball property implies k-ball property for all $k\geq 1$.
\item[(b)] k-ball property is equivalent to 3-ball property for all $k\geq 3$.
\item[(c)] A subspace $Y$ of $X$ has 3-ball property if and only if it is an M-ideal.
\item[(d)] 2-ball property does not imply 3-ball property.
\end{enumerate}
The following theorem is due to Lindenstrauss \cite{ls3}.
\begin{theorem}\label{tf7}
For a (real) Banach space $X$ and a closed subspace $Y$, TFAE:
\begin{enumerate}
\item[(i)] $Y$ has $\gamma$-ball property for all cardinals $\gamma$.
\item[(ii)] For each $x\notin Y$, each bounded linear map on $Y$ (and taking values in an arbitrary Banach space) has a norm-preserving extension to \textnormal{span}$[Y, x]$.
\end{enumerate}
\end{theorem}
The next property involving intersection of balls is motivated by Theorem \ref{te7} which was proved in the 
special case of finite dimensional Hilbert spaces. A complete proof of Kirszbraum's theorem asserting the validity 
of his theorem in the infinite dimensional setting is based on the following intersection property of closed balls 
in Hilbert spaces which we shall isolate as a definition in the class of (pairs of) Banach spaces.
\begin{definition}\label{fd5}
\textnormal{
Given a cardinal number $\gamma$ and Banach spaces X and Y, we shall say that the pair (X,Y) has $(K,\gamma)$- property if for all families of vectors $\{x_i\}_{i\in \wedge}~\text{ and}~ \{y_i\}_{i\in\wedge}$ in X and Y, respectively, with $\mid\wedge\mid=\gamma$ such that
$$ \| y_i-y_j\|\leq\| x_i-x_j\|,~\forall ~i,j\in \wedge ~\text{and}~\bigcap_{i\in \wedge}B(x_i,r_i)\neq\phi,$$
it follows that $ \bigcap_{i\in \wedge}B(y_i,r_i)\neq\phi$. We shall also say that the pair $(X,Y)$ has Kirszbraum 
property (property (K)) if it has $(K,\gamma)$-property for all cardinal numbers $\gamma$. Also, $X$ is said to 
have property (K) if the pair $(X,X)$ has it.}
\end{definition}
Let us note the following properties of property (K) and its relationship with other ball properties.
\begin{enumerate}
\item[(a)] $\mathbb{R}^n$ has property (K). In fact more is true: under the conditions of Definition \ref{fd5}, one has
 $ \mid\bigcap_{i \in \wedge}B(x_i,r_i) \mid \leq \mid\bigcap_{i\in \wedge}B(y_i,r_i)\mid$.
\item[(b)] $\gamma$- intersection property implies $(K,\gamma)$- property for all cardinals $\gamma$. In particular, injectivity implies property (K).
\item[(c)]	Hilbert spaces have property (K).
\end{enumerate}
In what follows, we shall include a proof of (c) and use it below to complete the proof of Kirszbraum's theorem 
which was sketched in Theorem \ref{te7} for the
 special case of finite dimensional Hilbert spaces.\\
Indeed, we show for a pair of Banach spaces $(X,Y)$ satisfying property (K) that if  $A\subset X~\text{ and ~if}\, f:A\rightarrow Y$ is a 1-Lipschitz map, then $f$ has a 1-Lipschitz extension to the whole of $X$. To this end, let $z \notin A$ and consider the family of balls
$$ \{B(x,\| x-z\|\}_{x\in A}~\text{and}~\{B(f(x),\| x-z\|\}_{x\in A}$$
in $X$ and $Y$, respectively. Since $\| f(x)-f(y)\| \leq \| x-y\| ~ \text{for~ all}~ x,y\in A$ 
and since $z$ obviously belongs 
to $B(x,\| x-z\|)~\text{for each}~ x\in A$, the given hypothesis shows that there exists $y\in 
B(f(x),\| x-z\|)~\text{for each}~ x\in A$. 
The assignment $x\rightarrow y$ , therefore, gives the desired extension of $f$  to  $A\bigcup \{z\}$. Finally, a 
little Zornification yields an extension of $f$ 
to a 1-Lipschitz mapping on the whole space $X$.
\begin{remark}\label{fr6}
(a): It is interesting to observe that property (K) is, in some sense, exclusive to Hilbert spaces.
 In fact, the implication $(i)\Rightarrow (ii)$ of Theorem \ref{te7} shows that as long as $Y=\ell_2$, the validity 
of property (K) for the pair $(X,\ell_2)$ yields that X is a Hilbert space. In fact, it turns out that one may 
consider any strictly convex space in place of $\ell_2$ to derive the same conclusion.\\
(b): As mentioned in Remark \ref{er9}, a metric analogue of Theorem \ref{te8} is also valid in the setting of 
X being a metric space. The appropriate property characterising 1-injectivity in this more general setting involves a certain 
ball-intersection property- which can be looked upon as the nonlinear analogue of the binary intersection property (BIP).
 The precise definition is given below:
\end{remark}
\begin{definition}\label{fd6'}
\textnormal{A metric space X is said to be {\it hyperconvex} (in the sense of Aronszajn and Panitchpakdi) if 
$\bigcap_{i\in \wedge}B(x_i,r_i) \neq \phi$ for any collection  in $\{x_i\}_{i\in\wedge}$ in X and  $\{r_i\}_{i\in\wedge} \subset \mathbb{R}^{+}$ such that $d(x_i,x_j) \leq 
r_i+r_j, \forall \,\, i,j \in  \wedge $.  }
\end{definition}
Here is the promised characterisation in which the notions of injectivity and absolute Lipschitz retract have to be understood in the sense of metric spaces.
\begin{theorem}\label{tf8}
For a metric space X, TFAE:
\begin{enumerate}
\item[(i)] X is injective.
\item[(ii)] X is hyperconvex.
\item[(iii)] X is a 1- absolute Lipschitz retract.
\end{enumerate}
\end{theorem}
A detailed proof of the above statement can be looked up in “R. Espinola and M. A. Khamsi, Introduction to hyperconvex spaces, 
Handbook of Metric Fixed Point Theory, Eds., W. A. Kirk and B. Sims, Kluwer Acd. Publishers, 2001”. I wish to thank Asuman Aksoy for drawing 
my attention to this reference, though in a different context.

We conclude by providing the details of proof of (c) above, following the ideas of Benjamini and Lindenstrauss \cite[Chapter 1]{bl}. 
The motivation for including a proof of (c), though well known in the literature, is to draw attention to the typically Hilbert space
 nature of the proof which seems to suggest the failure of property (K) in the Banach space setting, except perhaps in certain isolated trivial cases.

We begin the proof by noting that a Hilbert space being reflexive, its closed balls are weakly compact and so, in the presence of finite intersection property,
 have a non-empty intersection. Thus it suffices to prove the assertion for finite collections of balls in the same Hilbert space H which may even 
be assumed to be finite dimensional as only finite intersections are involved.

Thus, let $x\in \bigcap_{i\in \wedge} B(x_i,r_i)$ and assume, without loss of generality, that $x\neq x_i,~ for~ all~ i\geq1$. Consider the function $f:H\rightarrow R$ given by
$$ f(z)=\max_{i\leq n}\frac{||z-y_i||}{||x-x_i||}.$$
Being continuous on the finite dimensional space H with $f(z)\rightarrow\infty~ as~ z\rightarrow\infty$, it 
follows that $f$ achieves its (absolute) minimum , 
say $\lambda$, at some point $y$ in H.

{Claim:} $y\in \bigcap_{i\in \wedge} B(y_i,r_i)$\\
It clearly suffices to show that $\lambda\leq 1$. Denote by J the set
$$J=\{1\leq i\leq n;\| y-y_i\| =\lambda \| x-x_i\| \}.$$
Let us note that $y\in A=\text{conv}\{y_i\,;i\in J\}$. Indeed, otherwise, there exists a hyperplane L separating $y$ from A.  Thus it is possible to choose $y_0\neq y$ in H such that $\| y_0-y_i \| < \| y-y_i \| =\lambda \| x-x_i\|~ for ~i\in J$. Also, because $\| y_0-y_i\|\leq\| y-y_i\| <\lambda\| x-x_i\|~ for~ i\in J$, it follows that $f(y_0)<f(y),$ contradicting that $f(y)$ is minimal. Thus we can choose scalars $\{\alpha_i\}_{i\in J}$ such that $y=\sum \alpha_i  y_i,\alpha_i\geq 0,\sum \alpha_i =1$. Now for $i,j\in J$, we have
\begin{align*}
\| x_i-x_j\|^2 &\geq \| y_i-y_j\|^2\\ &= \| y-y_i\|^2 + \| y-y_j\|^2 -2\langle y-y_i,y-y_j\rangle\\ &=\lambda^2\| x-x_i\|^2 +\lambda^2\| x-x_j\|^2 -2\langle y-y_i,y-y_j\rangle.
\end{align*}
Multiplying the above inequality on both sides by $\alpha_i \alpha_j$  , summing over $i,j\in J$ and making use of 
the fact that
$$ \sum_{i,j\in J}\alpha_i  \alpha_j \langle y-y_i,y-y_j\rangle=0 $$
and
$$ 2\| \sum_{i\in J} \alpha_i x_i\|^2=\sum_{i,j\in J}\alpha_i \alpha_j (\| x_i \|^2 + 
\| x_j\|^2-\| x_i-x_j\|^2)$$
we see that

\begin{eqnarray*}
0&\geq& \sum_{i,j\in J} \alpha_i\alpha_j (\lambda^2\| x- x_i \|^2 + \lambda^2\| x- 
x_j\|^2-\| x_i-x_j\|^2)\\
 &=& \sum_{i,j\in J} \alpha_i \alpha_j (\| x- x_i \|^2 + \| x- 
x_j\|^2-\| x_i-x_j\|^2)\\
 &+&(\lambda^2 -1)\sum_{i,j\in J} \alpha_i \alpha_j(\| x- x_i \|^2 
+ \| x- x_j\|^2.
\end{eqnarray*}
Finally, noting that the first term in the above sum is non-negative by virtue of $(\ast)$, it follows that  $\lambda^2-1\leq 0, i.e., \lambda\leq1$ and this completes the proof.\\


{\bf Acknowledgement.} Work on this paper was initiated during the author’s visit at ISI Delhi and at IISc Bangalore where he spent part of the sabbatical leave from his university from Aug. 2012 to March 2013. The author wishes to thank his hosts Prof. Ajit Iqbal Singh (ISI Delhi) and Prof. Gadadhar Misra(IISc Bangalore) for their kind invitation and hospitality during the author’s visit to these institutions. He would also like to thank his home university for allowing him to avail of sabbatical leave during the indicated period.

\bibliographystyle{amsplain}

\end{document}